\documentclass{article}
\usepackage{tikz,xcolor,hyperref}
\usepackage{tikz}
\usepackage{amsmath,amscd}
\usepackage{amssymb}
\usepackage{theorem}
\usepackage{xcolor} % ajouté par Marie-Pierre

\newtheorem{theorem}{Theorem}
\newtheorem{proposition}[theorem]{Proposition}
\newtheorem{corollary}[theorem]{Corollary}
\newtheorem{lemma}[theorem]{Lemma}

{\theorembodyfont{\rmfamily}%
  \newtheorem{example}[theorem]{Example}

\newenvironment{proof}{\noindent\textit{Proof.}}
{\QED\vskip\theorempostskipamount} 

\def\petitcarre{\vrule height4pt width 4pt depth0pt}
\def\QED{\relax\ifmmode\eqno{\hbox{\petitcarre}}\else{%
  \unskip\nobreak\hfil\penalty50\hskip2em\hbox{}\nobreak\hfil
  \petitcarre
  \parfillskip=0pt \finalhyphendemerits=0\par\smallskip}
  \fi}

\newcommand{\Z}{\mathbb{Z}}
\newcommand{\N}{\mathbb{N}}
\newcommand{\cL}{\mathcal L}
\newcommand{\A}{\mathcal A}
\newcommand{\B}{\mathcal B}
\newcommand{\CC}{\mathcal C}
\DeclareMathOperator{\Card}{Card}
\newcommand{\edge}[1]{\stackrel{#1}{\rightarrow}}

\newcommand{\ie}{{ that is, }}

\definecolor{lime}{HTML}{A6CE39}
\DeclareRobustCommand{\orcidicon}{%
	\begin{tikzpicture}
	\draw[lime, fill=lime] (0,0)
	circle [radius=0.16]
	node[white] {{\fontfamily{qag}\selectfont \tiny ID}};
	\draw[white, fill=white] (-0.0625,0.095)
	circle [radius=0.007];
	\end{tikzpicture}
	\hspace{-2mm}
}
\foreach \x in {A, ..., Z}{%
 \expandafter\xdef\csname orcid\x\endcsname{\noexpand%
 \href{https://orcid.org/\csname orcidauthor\x\endcsname}{\noexpand\orcidicon}}
}
\title{Unambiguously coded shifts}
\author{Marie-Pierre B\'eal\orcidA{} and Dominique Perrin and Antonio Restivo}
\begin{document}
\maketitle
\begin{abstract}
  We study the coded shifts introduced by Blanchard and
  Hansel \cite{BlanchardHansel1986}. We give several
  constructions which allow one to represent a coded shift
  as an unambiguous one.
\end{abstract}
%\tableofcontents
\section{Introduction}
%\marginpar{A faire:
% Montrer que c'est decidable pour un shift minimal $X$
%et un automate $\A$ de savoir si $\A$ est non-ambigu sur $X$.
%Il faut evidemment que $\cL(X)$ soit decidable.

%}
Coded shifts were introduced in \cite{BlanchardHansel1986}
as a generalization of irreducible sofic shifts. A shift
space $X$ is said to be coded by a prefix code $C$
if the factors of $X$ are the factors of $C^*$ (more explicit
definitions are given below). Recently
some interest has appeared for those coded shifts
which are unambiguously coded in
\cite{Pavlov2018} and \cite{BurrDasWolfYang2020}.
To be unambiguously coded by $C$
means that every infinite sequence of the shift $X$
has at most one factorization $\cdots c_{-1}c_0c_1\cdots$
with $c_n\in C$.

We investigate this notion and prove several results.
First of all, it follows from the work of Doris and Ulf-Rainer Fiebig
\cite{FiebigFiebig1992} that
every coded shift is unambiguously coded
(Theorem~\ref{theoremMain}).
This answers  a question raised in \cite{BurrDasWolfYang2020}.
Actually, only a weaker result is proved explicitly in
\cite{FiebigFiebig1992}, namely that every coded shift
can be recognized by a countable deterministic and co-deterministic
automaton \cite[Theorem 1.7]{FiebigFiebig1992}. We
reproduce here this result and its proof as Theorem~\ref{theoremFiebig1.7}.
We are indebted to Ulf-Rainer Fiebig for providing us a complete
proof of the stronger result, which is  indicated without
proof in \cite[Remark 1.8]{FiebigFiebig1992}. It is stated
here as Theorem~\ref{theoremFiebigRemark1.8} and
proved in full.

Unambiguously coded shifts can also be defined by automata
having the property of strong unambiguity.
We connect the notion of strong unambiguity with
that of recognizable morphism, which is central in the study
of shifts defined by morphisms (see~\cite{Queffelec2010}
or the more recent paper \cite{BertheSteinerThuswaldnerYassawi2019}). We develop this
aspect in a second paper \cite{BealPerrinRestivo2021}.

We also investigate synchronized shifts, which are defined by synchronized
prefix codes. We prove directly (that is, without using Theorem~\ref{theoremFiebigRemark1.8}) that every synchronized shift is unambiguously coded
(Theorem~\ref{theoremMain}). This allows us to prove that
every irreducible  sofic shift is unambiguously coded by a rational
prefix code (Corollary \ref{corollarySofic}).

We would like to thank Ulf-Rainer Fiebig for his
contribution and Francesco Dolce for reading the manuscript
and founding several mistakes.

\section{Languages and shift spaces}
Let $A$ be a finite alphabet. We denote by $A^*$ the set of words
on $A$ and by $\varepsilon$ the empty word. The
\emph{length} of a word $u$ is denoted $|u|$. A word $u$
is a \emph{factor} of $v$ if $v=pus$ for some words $p,s$.
It is  a \emph{prefix} of $v$ if $v=us$.
It is a \emph{proper prefix} if $s$ is nonempty
(that is, $u\ne v$).

A \emph{language} on the alphabet $A$ is a set of words on $A$.
For a language $U$, we denote by $U^*$ the set of (possibly empty)
words $u_1\cdots u_n$ with $u_i\in U$ and $n\ge 0$. A language
is \emph{rational} if it can be obtained from the subsets
of $A\cup\{\varepsilon\}$ by a finite number of unions,
set products and stars. 

An \emph{automaton} $\A$ on the alphabet $A$ is a graph
on a set $Q$ of vertices, called the \emph{states}
of $\A$ with edges labeled by $A$. Given two
sets $I,T$ of states called repectively the sets
of \emph{initial} and \emph{terminal} states,
the language \emph{recognized} by $\A$ is the set
of labels of paths from an element of $I$ to
an element of $T$. We denote $\A=(Q,i,T)$.

%{\color{red} injective labelling appelé local ?}

A \emph{deterministic automaton} on the alphabet $A$ is a set $Q$ with
a partial map $(q,a)\mapsto q\cdot a$ from $Q\times A$ to $Q$.
This map is extended to $(q,w)\mapsto q\cdot w$ by
associativity, that is $q\cdot wa=(q\cdot w)\cdot a$.
Thus a deterministic  automaton can be considered as a
particular case of automaton
with edges $p\edge{a}q$ whenever $p\cdot a=q$.

A co-determistic automaton is obtained from a deterministic one by
reverting the edges.

Given $i\in Q$ and $T\subset Q$, the determistic automaton \emph{recognizes}
the language $L=\{w\in A^*\mid i\cdot w\in T\}$.
%The elements of $Q$ are called \emph{states}, the state $i$
%the \emph{initial state} and $T$ the set of \emph{terminal states}.

An automaton is \emph{unambiguous}
if for every word $w$ and every pair of states $p,q$, there
is at most one path labeled $w$ from $p$ to $q$.
A deterministic automaton is unambiguous.

A \emph{code} on the alphabet $A$ is a set $C\subset A^+$ such that
every word on $A$ has at most one factorizations in words of $C$.
Formally, for every $n,m\ge 0$ and every $x_1,\ldots,x_n,y_1,\ldots,y_m$
in $C$,
the equality
\begin{displaymath}
  x_1\cdots x_n=y_1\cdots y_m
\end{displaymath}
implies $n=m$ and $x_i=y_i$ for $i=1,\ldots,n$.
The following well-known result relates unambiguous automata and codes.
A path from $i$ to $i$ in an automaton is \emph{simple} if
it does not pass by $i$ except at the origin and end of the path.
The following result is well known (see~\cite{BerstelPerrinReutenauer2009}).
\begin{theorem}
  Let $\A=(Q,i,i)$ be a strongly connected automaton and let $C$ be the set
  of labels of simple paths from $i$ to $i$. Then $\A$
  is unambiguous if and only if $C$ is a code.
  \end{theorem}

An automaton is \emph{strongly unambiguous} if the labelling of
bi-infinite paths is
injective, \ie it has at most one bi-infinite path with a given
bi-infinite label. A strongly connected automaton which
is strongly unambiguous is also unambiguous but the converse
is not true, as shown by the following example.

\begin{example}\label{exampleEven}
  Let $\A$ be the automaton represented in Figure~\ref{figureUnambiguous}.
  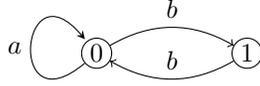
\begin{figure}[hbt]
    \centering
    \tikzset{node/.style={circle,draw,minimum size=0.4cm,inner sep=0.4pt}}
    \tikzset{title/.style={minimum size=0.5cm,inner sep=0.4pt}}
    \tikzstyle{every loop}=[->,shorten >=1pt,looseness=12]
    \tikzstyle{loop left}=[in=130,out=220,loop]
    \tikzstyle{loop right}=[in=-45,out=45,loop]
        \begin{tikzpicture}
          \node[node](0)at(0,0){$0$};\node[node](1)at(2,0){$1$};

          \draw[left,->,loop left,>=stealth](0)edge node{$a$}(0);
          \draw[above,->,bend left](0)edge node{$b$}(1);
          \draw[above,->,bend left](1)edge node{$b$}(0);
          %\draw[right,->,loop right,>=stealth](1)edge node{$b$}(1);
        \end{tikzpicture}
        \caption{An unambiguous automaton.}\label{figureUnambiguous}
  \end{figure}
  The automaton is unambiguous since it is deterministic. It is not
  strongly unambiguous because there are two bi-infinite paths
  labeled with letters $b$.
  %\begin{figure}[hbt]
%\centering
    %\tikzset{node/.style={circle,draw,minimum size=0.4cm,inner sep=0.2pt}}
    %\tikzset{title/.style={minimum size=0.5cm,inner sep=0.4pt}}
    %\begin{tikzpicture}
      %\node[title](h-2)at(-1,1){$\ldots$};
    %\node[node](h-1)at(0,1){};\node[node](h0)at(1,1){};\node[node](h1)at(2,1){};
      %\node[title](h2)at(3,1){$\ldots$};
      %\node[title](b-2)at(-1,0){$\ldots$};
    %\node[node](b-1)at(0,0){};\node[node](b0)at(1,0){};\node[node](b1)at(2,0){};
      %\node[title](b3)at(3,0){$\ldots$};

      %\draw[->,>= stealth](h-2)edge node{$a$}(h-1);
      %\draw[above,->,>= stealth](h-1)edge node{$a$}(h0);
      %\draw[above,->,>= stealth](h0)edge node{$a$}(h1);
      %\draw[left,->,>= stealth,bend left](h-1)edge node{$b$}(b-1);
      %\draw[left,->,>= stealth,bend left](b-1)edge node{$c$}(h-1);
      %\draw[left,->,>= stealth,bend left](h0)edge node{$b$}(b0);
      %\draw[left,->,>= stealth,bend left](b0)edge node{$c$}(h0);
      %%\draw[left,->,>= stealth,bend left](h1)edge node{$b$}(b1);
      %\draw[left,->,>= stealth,bend left](b1)edge node{$c$}(h1);
      %\draw[above,->,>= stealth](b-1)edge node{$a$}(b0);
      %\draw[above,->,>= stealth](b0)edge node{$a$}(b1);
    %\end{tikzpicture}
    %\caption{An unambiguous automaton.}\label{figureUnambiguous}
    %\end{figure}
\end{example}

The property of being strongly unambiguous can easily be tested
on a finite automaton because of the following property.
For an automaton $\A$ on the alphabet $A$
with $Q$ as set of states, its square is the automaton $\A\times\A$
with $Q\times Q$ as set of states and edges $(p,q)\edge{a}(r,s)$
if $p\edge{a}r$ and $q\edge{a}s$ are edges of $\A$.
Its non-diagonal part is the restriction to the
states of the form $(p,q)$ with $p\ne q$.
A strongly connected component of an automaton is trivial
if its set of edges is empty.
\begin{proposition}
  An unambiguous automaton $\A$ is strongly unambiguous if and only
  if the strongly connected components of the non-diagonal
  part of $\A\times \A$ are trivial.
\end{proposition}
\begin{proof}
  Let $p_n\edge{a_n}p_{n+1}$ and $q_n\edge{a_n}q_{n+1}$
  be two distinct bi-infinite paths with the same label.
  Let $E=\{n\in\Z\mid p_n\ne q_n\}$. Since $\A$ is unambiguous,
  the set $E$ is infinite.  Since $Q$ is finite,
  there exist $n,m\in E$ with $n<m$ such that $(p_n,q_n)=(p_m,q_m)$.
  Thus, $\A\times\A$ contains a cycle
  $(p_n,q_n)\xrightarrow{a_n\cdots a_{m-1}}(p_m,q_m)$.
  The converse is clear.
  \end{proof}
\begin{example}
  Let $\A$ be the automaton of Example~\ref{exampleEven}.
  The non-diagonal part of the
  automaton $\A\times \A$ is shown in Figure~\ref{figureUnambiguous2}.
  It is strongly connected in agreement with the
  fact that $\A$ is not strongly unambiguous.
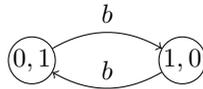
\begin{figure}[hbt]
    \centering
    \tikzset{node/.style={circle,draw,minimum size=0.4cm,inner sep=0.4pt}}
    \tikzset{title/.style={minimum size=0.5cm,inner sep=0.4pt}}
    \tikzstyle{every loop}=[->,shorten >=1pt,looseness=12]
    \tikzstyle{loop left}=[in=130,out=220,loop]
    \tikzstyle{loop right}=[in=-45,out=45,loop]
        \begin{tikzpicture}
          \node[node](0)at(0,0){$0,1$};\node[node](1)at(2,0){$1,0$};

          %\draw[left,->,loop left,>=stealth](0)edge node{$a$}(0);
          \draw[above,->,bend left](0)edge node{$b$}(1);
          \draw[above,->,bend left](1)edge node{$b$}(0);
          %\draw[right,->,loop right,>=stealth](1)edge node{$b$}(1);
        \end{tikzpicture}
        \caption{The square of $\A$.}\label{figureUnambiguous2}
  \end{figure}
  \end{example}

A language is \emph{recognizable} if it can be recognized by
a finite automaton. By Kleene's Theorem, a language
is recognizable if and only if it is rational (on all these notions,
see \cite{BerstelPerrinReutenauer2009}
or any textbook on formal languages).

For every language $L$, the \emph{minimal  automaton} of $L$
is the deterministic
automaton $\A(L)$ obtained as follows. For $u\in A^*$, denote
$u^{-1}L=\{v\in A^*\mid uv\in L\}$. The set $Q$ is the family of
nonempty sets $u^{-1}L$. Next, for $q=u^{-1}L$ and $a\in A$,
we define $q\cdot a=(ua)^{-1}L$ provided the right hand side is nonempty.
Then $\A(L)$ recognizes $L$ with the choice of $i=L$ and $T$
the family of sets $u^{-1}L$ containing $\varepsilon$.
A language is recognizable if and only if its minimal 
automaton is finite (actually, the minimal automaton
has the least possible number of states among all
deterministic automata recognizing $L$).

We consider the set $A^\Z$ of infinite two-sided sequences
of elements of $A$. It is a compact metric space for the distance
$d(x,y)=1/r(x,y)$ with
\begin{displaymath}
  r(x,y)=\min\{|n|\mid n\in\Z, x_n\ne y_n\}.
\end{displaymath}
The \emph{shift transformation} on $A^\Z$ is the map $S:A^\Z\to A^\Z$ defined
by $y=Sx$ if $y_n=x_{n+1}$ for all $n\in\Z$.

For a word  $u=u_0\ldots u_{p-1}\in A^*$ of length $p\ge 1$, we denote by $u^\infty$ the two-sided infinite
sequence $x\in A^\Z$ defined by $x_n=u_i$ whenever $i=n\bmod p$.
Such an element of $A^\Z$ is said to be a \emph{periodic point}.
For a sequence $(u_n)_{n\in \Z}$ of nonempty words, we denote by
\begin{displaymath}
  \cdots u_{-1}\cdot u_0u_1\cdots
\end{displaymath}
the two-sided infinite sequence $x$ such that
\begin{eqnarray*}
  \cdots x_{-2}x_{-1}&=&\cdots u_{-2}u_{-1}\\
  x_0x_1\cdots&=&u_0u_1\cdots.
\end{eqnarray*}
We also denote $w^\omega$ the one-sided infinite word
$www\cdots$, which is an element of $A^\N$ and by $^\omega w$ the
sequence $\cdots www$, which is an element of $A^{-\N}$.

A \emph{shift space} is a set $X$ of two-sided infinite sequences
on a finite alphabet $A$ which is closed and invariant by the shift
(see~\cite{LindMarcus2021} for the basic definitions of symbolic dynamics).

A shift space is a particular case of a \emph{dynamical system}
which is a pair $(X,T)$ formed of a compact metric space $X$
and a continuous invertible map $T:X\to X$.

If $X$ is a shift space, we denote by $\cL(X)$ the
\emph{language} of $X$, which is the set of finite factors
of the elements of $X$. It
follows from the definition that a shift space is defined by its language.
We denote by $\cL_n(X)$ the set of words of length $n$ in $\cL(X)$.

The language of a shift space $X$ is \emph{factorial} (that is, it
contains the factors of its elements) and extendable (that
is, for every $w\in \cL(X)$, there are letters $a,b\in A$
such that $awb\in\cL(X)$). Conversely, for every factorial
extendable language $L$, there is a shift space $X$ such that
$L=\cL(X)$.

A shift space $X$ on the alphabet $A$
is a shift \emph{of finite type} if there
is a finite set $W$ of words on $A$, such that $\cL(X)$
is the set of words without factor in $W$.

\begin{example}\label{exampleGolden}
  The set of sequences on $A=\{a,b\}$ without factor $bb$ is
  a shift space $X$ called the \emph{golden mean shift}.
  One has $\cL(X)=\{a,ba\}^*\{\varepsilon,b\}$.
  \end{example}

A shift space $X$ is \emph{irreducible} if for every
$u,v\in\cL(X)$ there exists a word $w$ such that $uwv\in\cL(X)$.
It is \emph{minimal} if it does not contain properly a nonempty
shift space. A minimal shift is irreducible but the converse is false.
For example, the golden mean shift is irreducible, but not minimal
since it contains $a^\Z$.

A shift space $X$ is called \emph{sofic} if $\cL(X)$ is a rational
language. As an equivalent definition, $X$ is sofic if
it is recognized by a finite automaton $\A$,
that is, it is the set of labels of two-sided infinite paths in
$\A$.
   \begin{figure}[hbt]
    \centering
    \tikzset{node/.style={circle,draw,minimum size=0.4cm,inner sep=0.4pt}}
    \tikzset{title/.style={minimum size=0.5cm,inner sep=0.4pt}}
    \tikzstyle{every loop}=[->,shorten >=1pt,looseness=12]
    \tikzstyle{loop left}=[in=130,out=220,loop]
    \tikzstyle{loop right}=[in=-45,out=45,loop]
    \begin{tikzpicture}
%golden
     \node[node](0)at(0,0){$0$};\node[node](1)at(2,0){$1$};

          \draw[left,->,loop left,>=stealth](0)edge node{$a$}(0);
          \draw[above,->,bend left](0)edge node{$b$}(1);
          \draw[above,->,bend left](1)edge node{$a$}(0);

    %even
          \node[node](0)at(5,0){$0$};\node[node](1)at(7,0){$1$};

          \draw[left,->,loop left,>=stealth](0)edge node{$a$}(0);
          \draw[above,->,bend left](0)edge node{$b$}(1);
          \draw[above,->,bend left](1)edge node{$b$}(0);
          %\draw[right,->,loop right,>=stealth](1)edge node{$b$}(1);
        \end{tikzpicture}
        \caption{The golden mean shift and the even shift.}\label{figureEven}
   \end{figure}
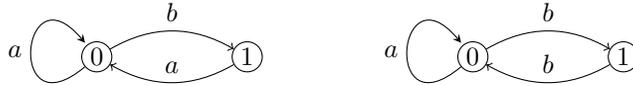
   \begin{example}
     The golden mean shift is sofic and it is recognized by
     the finite automaton of Figure~\ref{figureEven} on the left.
     
  The set $X$ of two-sided sequences on $\{a,b\}$ such that
  the number of consecutive $b$ between two $a$ is even
  is an irreducible sofic shift. 
This shift space, called the \emph{even shift}, is also
  recognized by the automaton of Figure~\ref{figureEven}
  on the right.
  We have $\cL(X)=\{a,bb\}^*\{\varepsilon,b\}$.
  The minimal automaton of $\cL(X)$ is shown in Figure~\ref{figureAutomate}
  (the state $0$ is initial and all states are terminal).
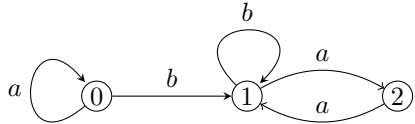
\begin{figure}[hbt]
    \centering
    \tikzset{node/.style={circle,draw,minimum size=0.4cm,inner sep=0.4pt}}
    \tikzset{title/.style={minimum size=0.5cm,inner sep=0.4pt}}
    \tikzstyle{every loop}=[->,shorten >=1pt,looseness=12]
    \tikzstyle{loop left}=[in=130,out=220,loop]
    \tikzstyle{loop above}=[in=45,out=135,loop]
    \begin{tikzpicture}
      \node[node](0)at(-2,0){$0$};
          \node[node](1)at(0,0){$1$};
          \node[node](2)at(2,0){$2$};
\draw[left,->,>=stealth,loop left](0)edge node{$a$}(0);
\draw[above,->,loop above,>=stealth](1)edge node{$b$}(1);
\draw[above,->,>=stealth](0)edge node{$b$}(1);
          \draw[above,->,bend left](1)edge node{$a$}(2);
          \draw[above,->,bend left](2)edge node{$a$}(1);
          %\draw[right,->,loop right,>=stealth](1)edge node{$b$}(1);
        \end{tikzpicture}
        \caption{The minimal automaton of $\cL(X)$.}\label{figureAutomate}
   \end{figure}
  
\end{example}

For any finite graph $G$, the \emph{edge shift} $X_G$ on $G$
is the set of biinfinite paths in $G$. It is a shift
of finite type. Conversely,  every shift of finite
type is the set of labels of bi-infinite paths in
 a strongly unambiguous automaton $\A$.

Let $X,Y$ be shift spaces on alphabets $A,B$ respectively.
Given   $m,n\ge 0$, a map $f:\cL_{n+m+1}(X)\to B$
is called a $n+m+1$-\emph{block map}. The \emph{sliding block code}
defined by $f$ is the map $\varphi:X\to B^\Z$ defined by
$y=\varphi(x)$ if $y_i=f(x_{i-m}\cdots x_{i+n})$ for all $i\in\Z$.
The simplest case corresponds to $n=m=0$. Such a code
is called a $1$-block code.
A \emph{factor map} $f:X\to Y$ is a map defined by a sliding block code
from $X$ onto $Y$.

A factor map $\varphi:X\to Y$ which is one-to-one is called
a \emph{conjugacy}.

\begin{example}\label{exampleFactorGolden}
  Consider the golden mean shift $X$. Let
  $f:\cL_2(X)\to \{a,b\}$ be the block map $aa\to a,ab\to b,ba\to b$. The
  sliding block code defined by $f$ is a factor map from $X$ onto the
  even shift.
\end{example}

Let $X$ be a shift space on $A$.  Recall that a point $x\in X$
  is \emph{periodic} if $S^n(x)=x$ for some $n\ge 1$. Otherwise,
  it is \emph{aperiodic}. A shift is aperiodic if it
  does not contain periodic points. For a minimal shift,
  being aperiodic is equivalent to being infinite.
  %%%%%%%%%%%%%%%%%%%%%%%%%%%%%
\section{Morphisms}
Let $\varphi:A^*\to A^*$ be a morphism. The \emph{language} of $\varphi$,
denoted $\cL(\varphi)$, is the set of factors of the words $\varphi^n(a)$
for $n\ge 0$ and $a\in A$. The \emph{shift} generated by $\varphi$,
denoted $X(\varphi)$ is the set of sequences with all their factors
in $\cL(\varphi)$.

\begin{example}
  The morphism $\varphi:a\mapsto ab,b\mapsto a$ is called the
  \emph{Fibonacci morphism}. The corresponding shift $X(\varphi)$
  is called the \emph{Fibonacci shift}.
  
\end{example}

A morphism $\varphi:A^*\to A^*$ is \emph{primitive} if there is
an $n\ge 1$ such that every $a\in A$ appears in every $\varphi^n(b)$
for $b\in A$.

The following result is well known (see~\cite{Queffelec2010} for example).
\begin{theorem}
For every primitive morphism $\varphi$, the shift $X(\varphi)$ is minimal.
  \end{theorem}
The automaton $\A(\varphi)$ associated to a morphism $\varphi:B^*\to A^*$ is a bouquet
of circles (that is,
a union of cycles sharing all the same vertex) labeled $\varphi(b)$ for $b\in B$. It has vertices
\begin{displaymath}
  \{(b,i)\mid b\in B,\ 0< i<|\varphi(b)|\}\cup\{\omega\}
\end{displaymath}
where $\omega$ is the vertex common to all cycles and edges
\begin{displaymath}
  E(\varphi)=\{[b,i]\mid b\in B,\ 0\le  i<|\varphi(b)|\}.
\end{displaymath}
The label of the edge $[b,i]$ is $a_i$ if $\varphi(b)=a_0\cdots a_k$. Its
source is
\begin{displaymath}
  s([b,i])=\begin{cases}(b,i)&\mbox{if $i\ne 0$}\\
  \omega&\mbox{otherwise}.
  \end{cases}
\end{displaymath}
Its range is
\begin{displaymath}
  r([b,i])=\begin{cases}(b,i+1)&\mbox{if $i+1<|\varphi(b)|$}\\
  \omega&\mbox{otherwise}.
  \end{cases}
\end{displaymath}

\begin{example}
  The automaton associated with the Thue-Morse morphism is the
  automaton of Figure~\ref{figureAutomatonThueMorse}.
\end{example}

\begin{example}
  The automaton associated with the Fibonacci morphism is the golden mean
  automaton of Figure~\ref{figureFiboAutomaton} on the left.
     \begin{figure}[hbt]
    \centering
    \tikzset{node/.style={circle,draw,minimum size=0.4cm,inner sep=0.4pt}}
    \tikzset{title/.style={minimum size=0.5cm,inner sep=0.4pt}}
    \tikzstyle{every loop}=[->,shorten >=1pt,looseness=12]
    \tikzstyle{loop left}=[in=130,out=220,loop]
    \tikzstyle{loop right}=[in=-45,out=45,loop]
    \begin{tikzpicture}
%golden
     \node[node](0)at(0,0){$0$};\node[node](1)at(2,0){$1$};

          \draw[left,->,loop left,>=stealth](0)edge node{$a$}(0);
          \draw[above,->,bend left](0)edge node{$a$}(1);
          \draw[above,->,bend left](1)edge node{$b$}(0);

        \end{tikzpicture}
        \caption{The automaton associated with the Fibonacci shift.}\label{figureFiboAutomaton}
   \end{figure}
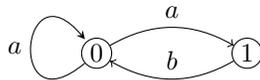
\end{example}

A morphism $\varphi:B^*\to A^*$ is \emph{circular} if it is injective and if for every $u,v\in A^*$
  \begin{displaymath}
    uv,vu\in\varphi(B^*)\Rightarrow u,v\in\varphi(B^*).
  \end{displaymath}
  The set $\varphi(B)$ is called a \emph{circular code}.
  There is a close connection between strong unambiguity
  and circular morphisms as shown by the following result
  (see~\cite{Restivo1975} for (i)$\Leftrightarrow$(ii)and
  \cite{Restivo1974} for (i)$\Leftrightarrow$ (iii)).
  \begin{theorem}
    Let $\varphi:B^*\to A^*$ be an injective morphism. The following conditions are equivalent.
    \begin{enumerate}
    \item[\rm(i)] $\varphi$ is circular.
      \item[\rm(ii)] The automaton $\A(\varphi)$
        is strongly unambiguous.
      \item[\rm(iii)] The closure under the shift of $\varphi(B^\Z)$ is of finite type.
        \end{enumerate}
  \end{theorem}
  \begin{example}
    The Fibonacci morphism is circular. Accordingly, the automaton $\A(\varphi)$
    represented in Figure~\ref{figureEven} is strongly unambiguous.
    The closure under the shift of $\varphi(\{a,b\}^\Z)$ is
    the golden mean shift.
    \end{example}
%%%%%%%%%%%%%%%%%%%%%%%%%%%%%%%%%%%%%
  \section{Relative unambiguity}
  Let $\A$ be a finite automaton on $A$ with $Q$ as set of states.
  Let $\Pi(\A)=\{(p_n,a_n,p_{n+1})\mid n\in\Z, p_n\edge{a_n}p_{n+1},p_n\in Q,a_n\in A\}$
  be the set of two-sided infinite paths in $\A$.
  The set $\Pi(\A)$ is a shift of finite type on the
  the alphabet $E$ formed by the edges of $\A$.
  Let $X$ be a subshift of $\Pi(\A)$.
  The automaton  is \emph{unambiguous on}  the shift $X$
  if for every sequence $y\in A^\Z$ there is at most one
  path in $X$ labeled $y$. Thus a strongly unambiguous automaton
  is the same as an automaton unambiguous on the 
  shift $\Pi(\A)$.
  The following statement shows that it is decidable whether
  an automaton is unambiguous on a sofic shift.
  \begin{proposition}\label{propositionUnambiguousOnSofic}
    Let $\A$ be a finite automaton on $A$ and let $X$
    be a sofic shift contained in $\Pi(\A)$.
    Let $\B$ be an automaton recognizing $X$.
    Let $\CC$ be the graph with edges $(p,q,r,s)\rightarrow(p',q',r',s')$
    where
    \begin{enumerate}
    \item[\rm(i)] $p\edge{e}p'$ and $q\edge{f}q'$ are edges of $\B$,
    \item[\rm(ii)] $e,f$ are edges of $\A$ of the form
      $r\edge{a}s$ and $r'\edge{a}s'$ for the same label $a\in A$.
    \end{enumerate}
    Then $\A$ is unambiguous on $X$ if and only if there
    is no biinfinite path in $\CC$ containing a vertex of the form
    $(p,q,r,s)$ with $r\ne s$.
  \end{proposition}
  \begin{proof}
    Assume that $\A$ is ambiguous on $X$. Let $x=(e_n)$ and $y=(f_n)$
    be two distinct elements of $X$ with the same label
    and thus such that $e_n:r_n\edge{a_n}r_{n+1}$
    and $f_n:s_n\edge{a_n}s_{n+1}$ in $\A$. Let $p_n\edge{e_n}p_{n+1}$
    and $q_n\edge{f_n}q_{n+1}$ be paths in $\B$. Since $x\ne y$, there
    is some $n_0$ such that $r_{n_0}\ne s_{n_0}$. Thus there
    is a biinfinite path $(p_n,q_n,r_n,s_n)$
    in $\CC$ containing the vertex $(p_{n_0},q_{n_0},r_{n_@},s_{n_0})$.
    The converse is clear.
    \end{proof}

  The previous construction takes a simpler
  form in the case of a shift of finite type.
  For an automaton $\A$ denote by $\lambda_\A(e)$
  the label of the edge $e$.
  \begin{proposition}
    Let $\A$ be a finite automaton on $A$ with $E$ as set of edges.
    Let $X$ be a  shift of finite type
    contained in $\Pi(\A$). Let $\B$ be a strongly unambiguous
    automaton recognizing $X$. Then $\A$ is unambiguous
    on $X$ if and only if the automaton $\mathcal C$ obtained
    from $\B$ by replacing every label $e\in E$ by $\lambda_\A(e)$
    is strongly unambiguous.
  \end{proposition}
  \begin{proof}
    Assume that $\mathcal C$ is strongly unambiguous.
    Since $\B$ is strongly unambiguous,
    two distinct elements of $X$ are the labels of distinct paths in $\B$.
    Thus, they cannot have the same label in $\A$.
    The converse is clear.
    \end{proof}

  \begin{example}
    Let $\A$ be the automaton of Figure~\ref{figureEven} on the right.
    Set $e=(0,a,0)$, $f=(0,b,1)$ and $g=(1,b,0)$. Let $X$
    be the shift of finite type recognized by the automaton $\B$ of Figure
    \ref{figureSuperEven} on the left. The automaton $\mathcal C$ obtained
    by replacing $e,f,g$ by their labels is represented on the right.
    It is strongly unambiguous and thus $\A$ is unambiguous on $X$.
       \begin{figure}[hbt]
    \centering
    \tikzset{node/.style={circle,draw,minimum size=0.4cm,inner sep=0.4pt}}
    \tikzset{title/.style={minimum size=0.5cm,inner sep=0.4pt}}
    \tikzstyle{every loop}=[->,shorten >=1pt,looseness=12]
    \tikzstyle{loop left}=[in=130,out=220,loop]
    \tikzstyle{loop right}=[in=-45,out=45,loop]
    \begin{tikzpicture}

    %even
          \node[node](0)at(5,0){$0$};\node[node](1)at(7,0){$1$};
\node[node](2)at(6,-1){$2$};
          \draw[left,->,loop left,>=stealth](0)edge node{$e$}(0);
          \draw[above,->,bend left](0)edge node{$f$}(1);
          \draw[above,->,bend left](1)edge node{$g$}(2);
          \draw[left,->,bend left](2)edge node{$e$}(0);
          %\draw[right,->,loop right,>=stealth](1)edge node{$b$}(1);

          %new
          \node[node](0)at(10,0){$0$};\node[node](1)at(12,0){$1$};
\node[node](2)at(11,-1){$2$};
          \draw[left,->,loop left,>=stealth](0)edge node{$a$}(0);
          \draw[above,->,bend left](0)edge node{$b$}(1);
          \draw[above,->,bend left](1)edge node{$b$}(2);
          \draw[left,->,bend left](2)edge node{$a$}(0);
        \end{tikzpicture}
        \caption{A shift of finite type.}\label{figureSuperEven}
   \end{figure}
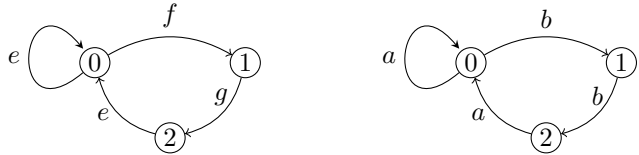
    \end{example}

  Let now $\varphi:B^*\to A^*$ be a morphism. We will show
  that for every shift space $X$ on $B$ there is a
  dynamical system $(X^\varphi,T)$ and a shift
  space $\alpha(X^\varphi)$ which is contained in the shift
  $\Pi(\A)$ of paths in the automaton $\A=\A(\varphi)$. Let
  for this
  \begin{equation}
    X^\varphi=\{(x,i)\mid x\in X,0\le i<|\varphi(x_0)|\}.
    \end{equation}
  We may consider on the space $X^\varphi$ the transformation
  \begin{displaymath}
    T(x,i)=\begin{cases}(x,i+1)&\mbox{if $i+1<|\varphi(x_0)|$}\\
    (S(x),0)&\mbox{otherwise.}
    \end{cases}
  \end{displaymath}
  In this way, $(X^\varphi,T)$ is a dynamical system.
  This system can actually be identified to a subshift of $\Pi(\A)$.
  Indeed, there is a unique map $\alpha:X^\varphi\to E(\varphi)^\Z$ such
  that
  \begin{equation}
(\alpha(x,i))_0=[x_0,i]\label{eqalpha}
  \end{equation}
  and
  \begin{equation}\alpha\circ T=S\circ \alpha.\label{eqIntertwin}
  \end{equation}
  Note that the map $\alpha$ is a $1$-block code.
  It is a homeomorphism
  from $X^\varphi$ into $\Pi(\A)$ and \eqref{eqIntertwin}
  shows that it identifies $(X^\varphi,T)$ with a subshift of $\Pi(\A)$.
  \begin{example}
\begin{figure}[hbt]
    \centering
    \tikzset{node/.style={circle,draw,minimum size=0.4cm,inner sep=0.4pt}}
    \tikzset{title/.style={minimum size=0.5cm,inner sep=0.4pt}}
    \tikzstyle{every loop}=[->,shorten >=1pt,looseness=12]
    \tikzstyle{loop left}=[in=130,out=220,loop]
    \tikzstyle{loop right}=[in=-45,out=45,loop]
    \begin{tikzpicture}
%abba
      \node[node](0)at(-2,0){$(b,1)$};\node[node](1)at(0,0){$\omega$};
      \node[node](2)at(2,0){$(a,1)$};

      \draw[above,->,bend left,>=stealth](0)edge node{$a$}(1);
      \draw[above,->,bend left,>=stealth](1)edge node{$b$}(0);
          \draw[above,->,bend left](1)edge node{$a$}(2);
          \draw[above,->,bend left](2)edge node{$b$}(1);

        \end{tikzpicture}
    \caption{An automaton unambiguous on the Thue-Morse shift}
    \label{figureAutomatonThueMorse}
\end{figure}
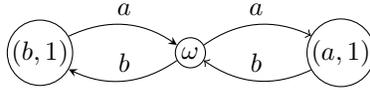
Let $\varphi:a\mapsto ab,b\mapsto ba$ be the Thue-Morse morphism.
The shift $X=X(\varphi)$ is called the \emph{Thue-Morse shift}.
The automaton $\A(\varphi)$ is represented in Figure~\ref{figureAutomatonThueMorse}. It is unambiguous on the shift $\alpha(X^\varphi)$.
Indeed, every word of length $5$ in $\cL(\varphi)$ contains
$aa$ or $bb$ and there is only one path labeled $aa$
or $bb$.
\end{example}

Let $X$ be a shift space on $B$. 
 A morphism $\varphi:B^*\to A^*$ is
 said to be \emph{recognizable on $X$} if for every $y\in A^\Z$,
 there is at most one pair $(x,k)$ with $x\in X$ and $0\le k<|\varphi(x_0)|$
 such that $y=S^k\varphi(x)$.

 Thus, in terms of the system $(X^\varphi,T)$, the morphism $\varphi$
 is recognizable on $X$ if and only if the map
 \begin{displaymath}
   \hat{\varphi}:(x,k)\mapsto S^k\varphi(x)
 \end{displaymath}
 is injective from $X^\varphi$ into $A^\Z$.

 This definition can also be formulated in terms of the unambiguity of the
 automaton $\A(\varphi)$ using the map $\alpha$ defined in Equation
 \eqref{eqalpha}.
 \begin{proposition}
   The morphism $\varphi:B^*\to A^*$ is recognizable on $X$
   if only if the automaton $\A(\varphi)$ is unambiguous
   on $\alpha(X^\varphi)$.
 \end{proposition}
 \begin{proof}
   This follows from the fact that for every $y\in A^\Z$,
   and $(x,k)\in X^\varphi$, one has $y=S^k\varphi(x)$
   if and only if $y$ is the label of $\alpha(x,k)$.
   \end{proof}

  The following important result plays a central role
  in the study of shifts defined by morphisms (see~\cite{Queffelec2010}
  or \cite{DurandPerrin2021}). A morphism is called \emph{aperiodic}
  if the shift $X(\varphi)$ is aperiodic.
  \begin{theorem}[Moss\'e, 1992]
    A primitive aperiodic morphism $\varphi$ is
    recognizable on the shift $X(\varphi)$.
  \end{theorem}

  \begin{example}
    The Fibonacci morphism is recognizable on $\{a,b\}^\Z$.
\end{example}
 \begin{example}
  The  Thue-Morse morphism
  is recognizable on the Thue-Morse shift.
 \end{example}

 Mosse's Theorem has been generalized to non primitive morphisms,
 as we shall see now.

  An automaton is unambiguous on $X$ \emph{for aperiodic points}
  if for every aperiodic $x\in X$ there is at most one path
  labeled $x$. A morphism $\varphi:B^*\to A^*$
  is recognizable on $X$ for aperiodic points if
  the automaton $\A(\varphi)$ is unambiguous on $X$
  for aperiodic points. The following result
  is from \cite{BertheSteinerThuswaldnerYassawi2019}.
  It generalizes Mosse's Theorem since, for an aperiodic morphism
  $\varphi$, the shift $X(\varphi)$ does not contain periodic points.
  
  \begin{theorem}[Berth\'e et al.,2019]
    A morphism $\varphi:A^*\to A^*$ is
    recognizable on $X(\varphi)$ for aperiodic points.
  \end{theorem}

  The notion of a recognizable morphism on $A^\Z$ for aperiodic points
  can be formulated as a problem of combinatorics on words. Indeed,
  $\varphi:B^*\to A^*$ is recognizable on $A^\Z$
  for aperiodic points if the restriction
  of $\varphi$ to $B$ is injective and every aperiodic sequence $x\in A^\Z$
  has a unique factorization in words of $\varphi(B)$.
  
  A morphism $\varphi:B^*\to A^*$ is \emph{indecomposable}
  if for every $\alpha:C^*\to A^*$ and $\beta:B^*\to C^*$
  such that $\varphi=\alpha\circ\beta$, one has
  $\Card(C)\ge\Card(B)$. In particular, if $\varphi$
  is recognizable, one has $\Card(A)\ge\Card(B)$.
  \begin{example}
    A morphism $\varphi:A^*\to A^*$ with $A=\{a,b\}$
    is indecomposable unless both words
    $\varphi(a),\varphi(b)$ are powers of the same word.
    In particular, the Thue-Morse morphism is indecomposable.
  \end{example}
  The following result 
  appears (using a different terminology) in~\cite{KarhumakiManuchPlandowski2003}.
  \begin{theorem}
    An indecomposable
    morphism $\varphi:B^*\to A^*$
    is recognizable on $A^\Z$ for aperiodic points.
    \end{theorem}

  \begin{example}
   The  Thue-Morse morphism
    is recognizable on $A^\Z$ for aperiodic points. Actually, the
    only periodic points with  more than one factorization are
    $(ab)^\infty$ and $(ba)^\infty$.
  \end{example}

  \begin{example}
    The morphism $a\mapsto ab,b\mapsto aa$ is the \emph{period-doubling
    morphism}. It is recognizable on $A^\Z$ for aperiodic points.
    The only point with more than one factorization is $a^\Z$.
    \end{example}

\section{Coded shifts}
The following definition appears in \cite{BlanchardHansel1986}.
A \emph{coded shift} is a shift space $X$ such that $\cL(X)$
is the set of factors of $C^*$
for some language $C$. We say that $X$ is
\emph{defined} or \emph{coded} by $C$.

Actually, a coded shift is defined in \cite{BlanchardHansel1986}
for a language $C$ which is a \emph{prefix code}, that is
such that $C$ does not contain any proper prefix of one its elements.
It is proved in~\cite[Proposition 2.5]{BlanchardHansel1986} that
this is not a restriction, in the sense that any coded
shift $X$ is defined by a prefix code $C$.

As an equivalent definition, a shift space is a coded shift
if there is a countable  strongly connected
graph $G$ with edges labeled by $A$
such that $X$ is the closure of the set of labels of bi-infinite paths in $G$
(see \cite{BlanchardHansel1986} proposition 2.1).
Such a graph is thus an automaton $\A$ with all states initial and final
and we also say that $\A$ recognizes $X$.

A countable automaton on a finite alphabet is sometimes
called a countable state \emph{Markov shift} (see~\cite{FiebigFiebig2002}).

When $C$ is a prefix code, the automaton can be taken to be
the minimal automaton $\A(C^*)$ of the set $C^*$. For this
automaton, the set of terminal states is reduced to $\{i\}$.

\begin{example}
  The even shift is a coded shift defined by the finite prefix code
  $C=\{b,aa\}$.
  \end{example}

A coded shift is irreducible. Indeed, if $u,v\in \cL(X)$, we have
$puq,rvs\in X^*$ for some words $p,q,r,s$. Then $puqrvs\in C^*$
and thus $uwv\in \cL(X)$ with $w=qr$.

A coded shift contains a dense set of periodic points. Indeed,
if $X$ is coded by $C$, then $u^\infty$ belongs to $X$
for every $u\in C^*$. As a consequence a coded shift
cannot be minimal unless it is periodic (that is equal to
the shifts of a periodic point).

Let $X$ be the coded shift defined by $C$. For every sequence $(c_n)_{n\in\Z}$
of elements of $C$, the sequence $\cdots c_{-1}\cdot c_0c_1\cdots$
belongs to $X$. Note that $X$ will contain in general other points.

The following result is well-known.
\begin{proposition}\label{propositionFactorCoded}
  A factor of a coded shift is a coded shift.
\end{proposition}
\begin{proof}
  Let $X$ be a coded shift on the alphabet $A$
  and let $\A$ be a countable strongly connected automaton
  with set of states $Q$ recognizing $X$.
  Let $\varphi:X\to Y$ be a factor map from $X$ onto $Y$ defined by
  a block map $f:\cL_{n+1}(X)\to B$. Let $\B$ be the automaton
  on the set of states $Q\times\cL_n(X)$ with edges
  $(p,bu)\edge{c}(q,ua)$ for all edges $p\edge{a}q$ of $\A$
  and $bua\in\cL_{n+1}(X)$ with $a,b\in A$ and $c=f(bua)$.
  Then $\B$ is a countable strongly connected automaton which
  recognizes $Y$.
\end{proof}
The same proof shows the well-known
fact that a factor of a sofic shift is sofic.
\begin{example}
  Let us apply the proof of Proposition~\ref{propositionFactorCoded} 
  to the factor map of Example~\ref{exampleFactorGolden}.
  We start from  the automaton of
  Figure \ref{figureEven2} on the left recognizing
  the golden mean shift.   We obtain
  as a result the automaton of Figure \ref{figureEven2}
  on the right which is identical to the automaton of
  the even shift of Figure~\ref{figureEven} on the right.
  \begin{figure}[hbt]
    \centering
    \tikzset{node/.style={circle,draw,minimum size=0.4cm,inner sep=0.4pt}}
    \tikzset{title/.style={minimum size=0.5cm,inner sep=0.4pt}}
    \tikzstyle{every loop}=[->,shorten >=1pt,looseness=12]
    \tikzstyle{loop left}=[in=130,out=220,loop]
    \tikzstyle{loop right}=[in=-45,out=45,loop]
    \begin{tikzpicture}
%golden
     \node[node](0)at(0,0){$0$};\node[node](1)at(2,0){$1$};

          \draw[left,->,loop left,>=stealth](0)edge node{$a$}(0);
          \draw[above,->,bend left](0)edge node{$b$}(1);
          \draw[above,->,bend left](1)edge node{$a$}(0);

          %even
          \tikzstyle{every loop}=[->,shorten >=1pt,looseness=6]
          \tikzstyle{loop left}=[in=140,out=210,loop]
          \node[node](0)at(5,0){$(0,a)$};\node[node](1)at(7,0){$(1,b)$};

          \draw[left,->,loop left,>=stealth](0)edge node{$a$}(0);
          \draw[above,->,bend left](0)edge node{$b$}(1);
          \draw[above,->,bend left](1)edge node{$b$}(0);
          %\draw[right,->,loop right,>=stealth](1)edge node{$b$}(1);
        \end{tikzpicture}
        \caption{The golden mean shift and the even shift.}\label{figureEven2}
   \end{figure}
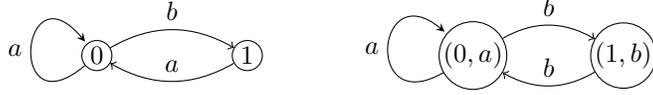
  \end{example}

\begin{corollary}
  The class of coded shifts is closed under conjugacy.
\end{corollary}

An interesting class of coded shifts is formed by the \emph{$\beta$-shifts}.
Let us recall briefly the definition (see~\cite{FrougnySakarovitch2010}
or also \cite{Blanchard1989} for a detailed presentation). Let $\beta$ be a real number with $\beta>1$.
We define for every real number $t\in[0,1]$ its $\beta$-expansion
as the sequence $d_\beta(t)=(x_n)_{n\ge 1}$ of integers $<\beta$ such that
\begin{displaymath}
  t=\sum_{n\ge 1}x_n\beta^{-n}
\end{displaymath}
which is lexicographically maximal. Thus $x_n$ is obtained
by setting $x_1=\lfloor \beta t\rfloor$ (with $\lfloor x\rfloor$
denoting the integer part of $x$) with $r_1=\beta t-x_1$
and recursively $x_n=\lfloor \beta r_{n-1}\rfloor$ with
$r_n=\beta r_{n-1}-x_n$.

For example, if $\beta=(1+\sqrt{5})/2$ is the golden mean, the
$\beta$-expansion of $1$ is $d_\beta(1)=110000\cdots$.

The \emph{$\beta$-shift} $X_\beta$ defined by $\beta$ is the closure of the
set of expansions of all numbers in $[0,1)$.
The $\beta$-shift is a coded shift recognized by the automaton
of Figure~\ref{figureBetaShift} where the sequence
$g_n$, called the generating sequence of $\beta$,
is either the expansion of $1$ if it is not evenutally $0$, or
the sequence $(x_1\cdots x_{k-1}(x_k-1))^\omega$ if $d_\beta(1)=x_1\cdots x_k00\cdots$
with $x_k\ne 0$.
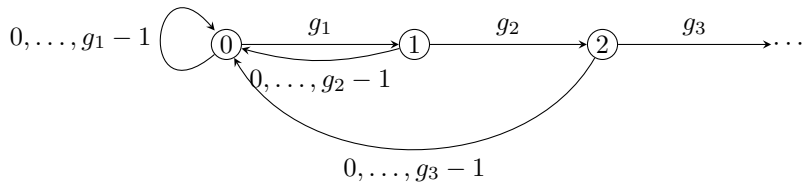
\begin{figure}[hbt]
    \centering
    \tikzset{node/.style={circle,draw,minimum size=0.4cm,inner sep=0.4pt}}
    \tikzset{title/.style={minimum size=0.5cm,inner sep=0.4pt}}
    \tikzstyle{every loop}=[->,shorten >=1pt,looseness=12]
    \tikzstyle{loop left}=[in=130,out=220,loop]
    \tikzstyle{loop right}=[in=-45,out=45,loop]
    \begin{tikzpicture}
      \node[node](0)at(0,0){$0$};
      \node[node](1)at(2.5,0){$1$};
      \node[node](2)at(5,0){$2$};
      \node[title](3)at(7.5,0){$\cdots$};

      \draw[->,left,loop left,>=stealth](0)edge node{$0,\ldots,g_1-1$}(0);
      \draw[->,above,>=stealth](0)edge node{$g_1$}(1);
      \draw[->,above,>=stealth](1)edge node{$g_2$}(2);
      \draw[->,above,>=stealth](2)edge node{$g_3$}(3);
      \draw[->,bend left=15,below,>=stealth](1)edge node{$0,\ldots,g_2-1$}(0);
      \draw[->,bend left=60,below,>=stealth](2)edge node{$0,\ldots,g_3-1$}(0);
      \end{tikzpicture}
\caption{An automaton recognizing the $\beta$-shift.}\label{figureBetaShift}
\end{figure}
Note that for the case where $\beta$ is the golden mean, we obtain
the automaton of Figure~\ref{figureBetaShift2} which recognizes
the golden mean shift.
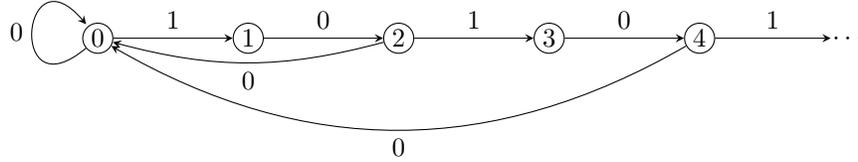
\begin{figure}[hbt]
    \centering
    \tikzset{node/.style={circle,draw,minimum size=0.4cm,inner sep=0.4pt}}
    \tikzset{title/.style={minimum size=0.5cm,inner sep=0.4pt}}
    \tikzstyle{every loop}=[->,shorten >=1pt,looseness=12]
    \tikzstyle{loop left}=[in=130,out=220,loop]
    \tikzstyle{loop right}=[in=-45,out=45,loop]
    \begin{tikzpicture}
      \node[node](0)at(0,0){$0$};
      \node[node](1)at(2,0){$1$};
      \node[node](2)at(4,0){$2$};
      \node[node](3)at(6,0){$3$};
      \node[node](4)at(8,0){$4$};
      \node[title](5)at(10,0){$\cdots$};

      \draw[->,left,loop left,>=stealth](0)edge node{$0$}(0);
      \draw[->,above,>=stealth](0)edge node{$1$}(1);
      \draw[->,above,>=stealth](1)edge node{$0$}(2);
      \draw[->,above,>=stealth](2)edge node{$1$}(3);
      \draw[->,above,>=stealth](3)edge node{$0$}(4);
      \draw[->,above,>=stealth](4)edge node{$1$}(5);
      %\draw[->,bend left,below,>=stealth](1)edge node{$0,\ldots,g_2-1$}(0);
      \draw[->,bend left=15,below,>=stealth](2)edge node{$0$}(0);
      \draw[->,bend left=30,below,>=stealth](4)edge node{$0$}(0);
      \end{tikzpicture}
\caption{An automaton recognizing the golden mean shift.}\label{figureBetaShift2}
\end{figure}

A automaton is called \emph{reversible} if it is both deterministic
and co-deterministic. Note that the set labels of simple
paths around a state in a reversible  automaton
 is not only
a prefix code, but actually a bifix code, that is such that
its reversal is also a prefix code.

The following result is from \cite{FiebigFiebig1992} (Theorem 1.7).
\begin{theorem} \label{theoremFiebig1.7}
For every coded shift $X$, the set $\cL(X)$ is recognized by a countable
strongly connected automaton which is reversible.
\end{theorem}
\begin{proof}
Since $X$ is a coded shift, the set $\cL(X)$ is recognized by a
countable strongly connected automaton. Let $p$ be a state of this
automaton. Let us assume that all infinite paths starting at $p$ have
the same label. Then since the automaton is strongly connected, the
shift $X$ is coded by a code containing a single word and the result
is trivial. 
Thus we may assume that there are finite words $ya$ and $yb$, where $a,b$ are
letters with $a \neq b$, labelling paths starting at $p$. Without
loss of generality, we may assume that the paths labeled $ya,yb$
share the same initial part labeled $y$.
Similarly
there are finite words $ct$, $dt$, where $c,d$ are letters with $c \neq
d$, labelling paths ending in $p$. We may also
assume that the paths labeled $ct,dt$ share the same final
part. Since the automaton is strongly
connected there are words $u_1, u_2$ such that $yau_1ct$,
$ybu_2dt$ are labels of cycling paths around $p$. 
(see Figure~\ref{figureu_1u_2}).
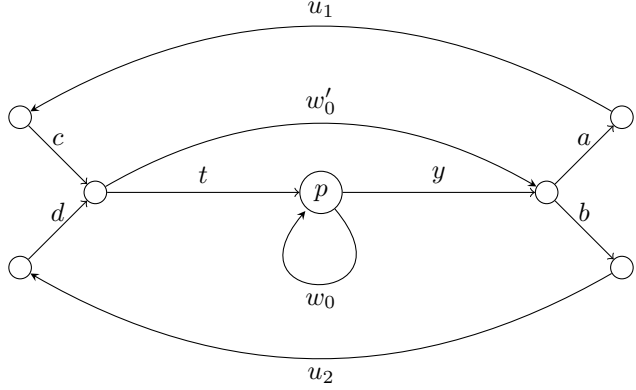
\begin{figure}[hbt]
\centering
\tikzset{node/.style={circle,draw,minimum size=0.3cm,inner sep=0.1cm}}
\tikzset{title/.style={minimum size=0cm,inner sep=0pt}}
\tikzstyle{every loop}=[->,shorten >=1pt,looseness=12]
\tikzstyle{loop above}=[in=50,out=130,loop]
\tikzstyle{loop below}=[in=230,out=310,loop]
\begin{tikzpicture}
  \node[node](c)at(0,1){};\node[node](d)at(0,-1){};
  \node[node](v)at(1,0){};%\node[node](vb)at(1,-.8){};
  \node[node](p)at(4,0){$p$};
  \node[node](u)at(7,0){};%\node[node](ub)at(7,-.8){};
  %\node[node](b)at(6,-.5){};
  \node[node](a)at(8,1){};\node[node](b)at(8,-1){};

  \draw[->,above](c)edge node{$c$}(v);\draw[->,above](d)edge node{$d$}(v);
  \draw[->,above](v)edge node{$t$}(p);%\draw[->,above](v)edge node{$v$}(p);
  \draw[->,above](p)edge node{$y$}(u);%\draw[->,above](p)edge node{$u$}(u);
  \draw[->,above](u)edge node{$a$}(a);\draw[->,above](u)edge node{$b$}(b);
  \draw[->,above,bend right,> = stealth](a)edge node{$u_1$}(c);
  \draw[->,above,bend left,> = stealth](v)edge node{$w'_0$}(u);
  \draw[->,below,bend left,> = stealth](b)edge node{$u_2$}(d);
  %\draw[->,below,bend right,> = stealth](v)edge node{$w'_1$}(u);
  %\draw[->,above,>=stealth](p)edge[loop above]node{$w_0$}(p);
  \draw[->,below,>=stealth](p)edge[loop below]node{$w_0$}(p);
\end{tikzpicture}
\caption{The words $u_1,u_2$.}\label{figureu_1u_2}
\end{figure}

Let $(w_i)_{i \in \mathbb{Z}}$ be an enumeration of the labels of all
finite paths from $p$ to $p$. Let $w'_i = tw_iy$ and $x=(x_i)_{i \in \mathbb{Z}}$ denote the
bi-infinite word
\begin{displaymath}
  \cdots au_1cw'_{-n} \cdots au_1cw'_{-1}au_1c \cdot w'_0au_1cw'_1
  \cdots au_1cw'_n \cdots.
\end{displaymath}
By construction the orbit of $x$ in $X$ is dense.
We define a new enumerable strongly connected automaton as follows. 
We start with the set of states $\mathbb{Z}$ and edges $(i, x_i,
i+1)$ (see Figure~\ref{figureNewAutomaton}).
\begin{figure}[hbt]
\centering
\tikzset{node/.style={circle,draw,minimum size=0.3cm,inner sep=0.1cm}}
\tikzset{title/.style={minimum size=0cm,inner sep=0pt}}
\tikzstyle{every loop}=[->,shorten >=1pt,looseness=12]
\tikzstyle{loop above}=[in=50,out=130,loop]
\tikzstyle{loop below}=[in=230,out=310,loop]
\begin{tikzpicture}
  \node[title](g)at(-3,0){$\cdots$};
  \node[node,inner sep=0.5pt](-1)at(-2,0){$-1$};\node[node](0)at(0,0){$0$};
  \node[node](1)at(2,0){$1$};\node[node](2)at(4,0){$2$};

  \draw[->,above](-1)edge node{$x_{-1}$}(0);
  \draw[->,above](0)edge node{$x_0$}(1);
  \draw[->,above](1)edge node{$x_1$}(2);
  \node[title](d)at(5,0){$\cdots$};
  \end{tikzpicture}
\caption{The new automaton.}\label{figureNewAutomaton}
  \end{figure}
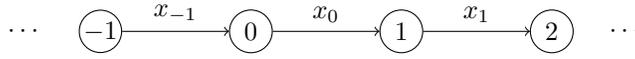
Let $p_n$ be the state obtained after reading $w'_n$ and $q_n$
the state before reading $w'_{-n}$. We add to the automaton a path
labelled by $bu_2d$ from $p_n$ to $q_n$ (see Figure~\ref{figureNewAutomaton2}).
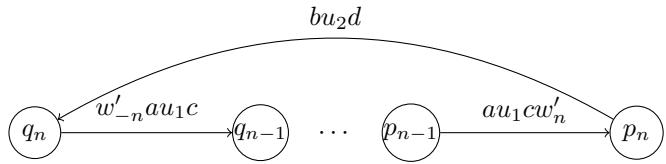
\begin{figure}[hbt]
\centering
\tikzset{node/.style={circle,draw,minimum size=0.3cm,inner sep=0.1cm}}
\tikzset{title/.style={minimum size=0cm,inner sep=0pt}}
\tikzstyle{every loop}=[->,shorten >=1pt,looseness=12]
\tikzstyle{loop above}=[in=50,out=130,loop]
\tikzstyle{loop below}=[in=230,out=310,loop]
\begin{tikzpicture}
  \node[node](qn)at(-4,0){$q_n$};\node[node,inner sep=0pt](qn-1)at(-1,0){$q_{n-1}$};
  \node[title](m)at(0,0){$\cdots$};
  \node[node,inner sep=0](pn-1)at(1,0){$p_{n-1}$};\node[node](pn)at(4,0){$p_n$};

  \draw[->,above](qn)edge node{$w'_{-n}au_1c$}(qn-1);
  \draw[->,above](pn-1)edge node{$au_1cw'_{n}$}(pn);
  \draw[->,bend right,above](pn)edge node{$bu_2d$}(qn);
   \end{tikzpicture}
\caption{Adding paths to the new automaton.}\label{figureNewAutomaton2}
  \end{figure}
By construction the automaton
is strongly connected, deterministic and co-deterministic. The set of
labels of its finite paths is $\cL(X)$.
\end{proof}
The above proof builds an automaton which is always infinite.
An important observation is that for some sofic shifts, such
automaton has to be infinite. Indeed, the languages $L$ recognized
by reversible finite automata belong to a particular
class of recognizable languages characterized by the
condition that for every $u,v,w\in A^*$, one has
\begin{displaymath}
  uvv^*w\subset L\Rightarrow uw\in L
\end{displaymath}
(see~\cite{Pin1992}). For example, the golden mean shift
$X$ cannot be recognized by a finite reversible automaton
since $\cL(X)$ obviously does not satisfy this condition.
\section{Unambiguously coded shifts}
A set $C\subset A^*$ is called a \emph{strong code}
if for every $x\in A^\Z$ there exists at most one pair
of a sequence $(c_n)_{n\in\Z}$ and an integer $k$ with $0\le k<|c_0|$
such that
\begin{displaymath}
  x=S^k(\cdots c_{-1}\cdot c_0c_1\cdots).
\end{displaymath}
As an equivalent formulation,  $C$ is a strong code
if for every $x\in A^\Z$ there is at most pair of a sequence
$(c_n)_{n\in\Z}$ and a factorization $c_0=ps$ with $s$ nonempty
such that
\begin{displaymath}
  x=\cdots c_{-2}c_{-1}p\cdot sc_1c_2\cdots
\end{displaymath}
In particular, the set $C$ has to be a circular code.
The following example, due to \cite{DevolderTimmermann1992}
shows that an infinite circular code need not be a strong code
(in the same paper, it is shown that the result is true
however if the set is recognizable).
\begin{example}
  The set $C=\{ab\}\cup\{ab^nab^{n+1}\mid n\ge 1\}$ is a circular code.
  It is not a strong code since $^\omega(ab)\cdot ab^2ab^3ab^4\cdots$ has two
  factorizations.
\end{example}

A shift space $X$ is said to be an \emph{unambiguously coded shift}
if it  is coded by some strong code $C$. In this case, we also say
that $X$ is \emph{unambiguously coded} by $C$.

This notion  appears in \cite{Pavlov2018} (where $C$ is called uniquely
decipherable when $X$ is unambiguously coded by $C$)
and also in \cite{BurrDasWolfYang2020} (were
$X$ is called uniquely representable if it is unambiguously
coded by some $C$).

\begin{example}\label{exampleEvenUnambiguous}
  The even shift is unambiguously coded. This is not true for
  $C=\{a,bb\}$ since the sequence $x=b^\infty$ has two factorizations.
  But it becomes true if we choose the prefix code $C'=(b^2)^*a$.
\end{example}
%{\color{red} Je suggère de supprimer ce théorème}.

The following result, which
is stronger than Theorem~\ref{theoremFiebig1.7},
is from \cite{FiebigFiebig1992} (Remark 1.8).
The proof is not given there but was kindly provided to us
by Ulf-Rainer Fiebig, from the notes of Doris Fiebig.
Theorem~\ref{theoremFiebig1.7} is a consequence of
Theorem~\ref{theoremFiebigRemark1.8} but we have
stated and proved it before, because its proof is much easier.
\begin{theorem} \label{theoremFiebigRemark1.8}
For every coded shift $X$, the set $\cL(X)$ is recognized by a countable
strongly connected automaton which is
strongly unambiguous as well as reversible.
\end{theorem}
\begin{proof}
Since $X$ is a coded shift, the set $\cL(X)$ is recognized by an
countable strongly connected automaton. The proof begins as that
of Theorem~\ref{theoremFiebig1.7}.
We may assume that there are finite words $ya$ and $yb$, where $a,b$ are
letters with $a \neq b$, labelling paths starting at some state $q$
of the automaton. Similarly
there are finite words $ct$, $dt$, where $c,d$ are letters with $c \neq
d$, labelling paths ending in $q$. Since the automaton is strongly
connected there are words $u_1, u_2$ such that $yau_1ct$,
$ybu_2dt$ are labels of cycling paths around $q$
(see Figure~\ref{figureu_1u_2Unambiguous}).
\begin{figure}[hbt]
\centering
\tikzset{node/.style={circle,draw,minimum size=0.4cm,inner sep=0.1cm}}
\tikzset{title/.style={minimum size=0cm,inner sep=0pt}}
\tikzstyle{every loop}=[->,shorten >=1pt,looseness=12]
\tikzstyle{loop above}=[in=50,out=130,loop]
\tikzstyle{loop below}=[in=230,out=310,loop]
\begin{tikzpicture}
  \node[node](c)at(0,1){};\node[node](d)at(0,-1){};
  \node[node,fill=red](v)at(1,0){$p$};%\node[node](vb)at(1,-.8){$p_3$};
  \node[node](p)at(4,0){$q$};
  \node[node,fill=blue](u)at(7,0){$r$};%\node[node](ub)at(7,-.8){$p_4$};
  %\node[node](b)at(6,-.5){};
  \node[node](a)at(8,1){};\node[node](b)at(8,-1){};

  \draw[->,above,> = stealth](c)edge node{$c$}(v);\draw[->,above,> = stealth](d)edge node{$d$}(v);
  \draw[->,above,> = stealth](v)edge node{$t$}(p);\draw[->,above,> = stealth](v)edge node{$t$}(p);
  \draw[->,above,> = stealth](p)edge node{$y$}(u);\draw[->,above,> = stealth](p)edge node{$y$}(u);
  \draw[->,above,> = stealth](u)edge node{$a$}(a);\draw[->,above,> = stealth](u)edge node{$b$}(b);
  \draw[->,above,bend right,> = stealth](u)edge node{$u$}(v);
  \draw[->,above,bend right,> = stealth](a)edge node{$u_1$}(c);
  \draw[->,below,bend right,> = stealth](v)edge node{$w$}(u);
  \draw[->,below,bend left=60,> = stealth](u)edge node{$v$}(v);
  \draw[->,below,bend left,> = stealth](b)edge node{$u_2$}(d);
  %\draw[->,above,>=stealth](p)edge[loop above]node{$w'_1$}(p);
  %\draw[->,below,>=stealth](p)edge[loop below]node{$w'_2$}(p);
\end{tikzpicture}
\caption{The words $u_1,u_2,u,v,w$.}\label{figureu_1u_2Unambiguous}
\end{figure}
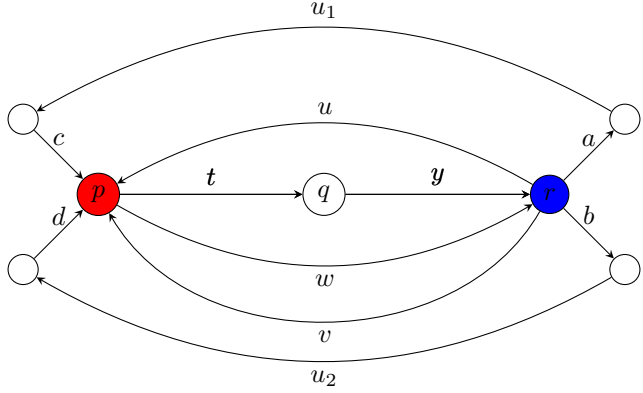
We now choose an enumeration 
 $(w'_i)_{i \ge 1}$  of the labels of all
finite paths from $q$ to $q$ with nonnegative indices, instead of
arbitrary ones. Let $w_{-i} = tw'_iy$, $w = ty$, $u =
au_1c$ and $v = bu_2d$. Note that the first and last letter of $u$ and
$v$ are distinct.

Inductively, we choose $m_i> 0$ such that
\begin{itemize}
\item the length of $(uw)^{m_1-1}$ is at least twice the length of $s_1=vw_{-1}v$,
\item the length of $(uw)^{m_2-1}$ is at least twice the length of
  $s_2=vw_{-2}(uw)^{m_1}s_1$,
  \item the length of $(uw)^{m_3-1}$ is at least twice the length of
    $s_3=vw_{-3}(uw)^{m_2}s_2$,
    \item and so on, where one always adds a word
      $vw_{-i-1}(uw)^{m_i}$ to the left.
 \end{itemize}

We denote by $x=(x_i)_{i \in \mathbb{Z}}$ the
bi-infinite word
$$\cdots
(uw)^{m_3}vw_{-3}(uw)^{m_2}vw_{-2}(uw)^{m_1}vw_{-1}u \cdot (wu)^\infty,
$$
where $x_0$ is equal to the first symbol of the right infinite periodic sequence $(wu)^\infty$.
It is the label in the graph of a path  shown below
\begin{displaymath}
  \cdots p\edge{w_{-2}}r\edge{(uw)^{m_1}}r\edge{v}p\edge{w_{-1}}r\edge{u}
  p\edge{wu}p\cdots
  \end{displaymath}
By construction the orbit of $x$ in $X$ is dense.

We define a new countable strongly connected automaton as follows. 
We start with the set of states $\mathbb{Z}$ and edges $(i, x_i,
i+1)$ (see Figure~\ref{figureNewAutomatonUnambiguous}).
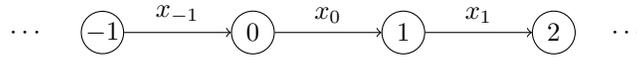
\begin{figure}[hbt]
\centering
\tikzset{node/.style={circle,draw,minimum size=0.3cm,inner sep=0.1cm}}
\tikzset{title/.style={minimum size=0cm,inner sep=0pt}}
\tikzstyle{every loop}=[->,shorten >=1pt,looseness=12]
\tikzstyle{loop above}=[in=50,out=130,loop]
\tikzstyle{loop below}=[in=230,out=310,loop]
\begin{tikzpicture}
  \node[title](g)at(-3,0){$\cdots$};
  \node[node,inner sep=0.5pt](-1)at(-2,0){$-1$};\node[node](0)at(0,0){$0$};
  \node[node](1)at(2,0){$1$};\node[node](2)at(4,0){$2$};

  \draw[->,above](-1)edge node{$x_{-1}$}(0);
  \draw[->,above](0)edge node{$x_0$}(1);
  \draw[->,above](1)edge node{$x_1$}(2);
  \node[title](d)at(5,0){$\cdots$};
  \end{tikzpicture}
\caption{The new automaton.}\label{figureNewAutomatonUnambiguous}
\end{figure}
Let $N_i < 0$ be the terminal vertex of the first $u$ in the factor
$(uw)^{m_i}vw_{-i}$ of $x$.
We choose an increasing sequence of integers $k_1 < k_2 < k_3 < \cdots$
with $k_1 = m_1$. For each $i >0$ we add a finite path labeled by $v$
from the positive terminal vertex $M_i$ of the path labeled by $(wu)^{k_i}w$
starting at the vertex $0$ to the vertex $N_i$ on the negative side
(see Figure~\ref{figureNewAutomaton2Unambiguous}).
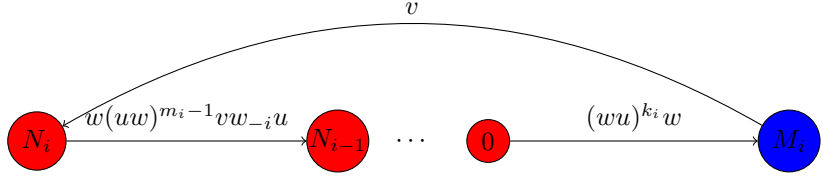
\begin{figure}[hbt]
\centering
\tikzset{node/.style={circle,draw,minimum size=0.3cm,inner sep=0.1cm}}
\tikzset{title/.style={minimum size=0cm,inner sep=0pt}}
\tikzstyle{every loop}=[->,shorten >=1pt,looseness=12]
\tikzstyle{loop above}=[in=50,out=130,loop]
\tikzstyle{loop below}=[in=230,out=310,loop]
\begin{tikzpicture}
  \node[node,fill=red](Ni)at(-6,0){$\small N_i$};
  \node[node,inner sep=0pt,fill=red](Ni-1)at(-2,0){$\small N_{i-1}$};
  \node[title]at(-1,0){$\cdots$};
  \node[node,fill=red](0)at(0,0){$0$};
  %\node[node,inner sep=0](Mi-1)at(2,0){$\small M_{i-1}$};
  \node[node,fill=blue](Mi)at(4,0){$\small M_i$};

  \draw[->,above](Ni)edge node{$w(uw)^{m_i-1}vw_{-i}u$}(Ni-1);
  \draw[->,above](0)edge node{$(wu)^{k_i}w$}(Mi);
  \draw[->,bend right,above](Mi)edge node{$v$}(Ni);
    %\draw[->,bend right,above](Mi-1)edge node{$v$}(Ni-1);

   \end{tikzpicture}
\caption{Adding paths to the new automaton.}\label{figureNewAutomaton2Unambiguous}
  \end{figure}
By construction the new automaton
is strongly connected, deterministic and co-deterministic. The set of
labels of its finite paths is $\cL(X)$.

We now show that the automaton is strongly unambiguous, \ie that it has at most one
bi-infinite path with a given label.

For later use, we record the following results.
\begin{itemize}
\item{(1a)} The word $uwb$ does not occur as a factor of $uwuw$ at any
  position.
  Indeed, inside $uwuw$ pairs of symbols at distance $|uw|$ agree. But
  the first symbol of $u$ is $a \neq b$.
\item{(1b)}  For the same reason the word $dwu$ cannot occur as a
  factor of $wuwu$ at any position.
\end{itemize}

For later use, we also note the following result.
\begin{itemize}
\item{(2)} The label of every path starting at $N_i$ begins with
  $w(uw)^{m_i-1}$ and there is no path starting at a vertex $k$ with
  $N_i < k <0$ whose label begins with $w(uw)^{m_i-1}$. This is due to
  (1a), (1b) and the definition of $m_i$ ensuring that $w(uw)^{m_i-1}$
  is long enough.
\end{itemize}
We now show that the automaton is strongly unambiguous. Let us assume that there are
two bi-infinite paths of the automaton $z \neq \bar{z}$ with the same
label.

Then
\begin{itemize}
  \item{(3)} One has $z_i \neq \bar{z_i}$ for all $i \in \mathbb{Z}$ since
    the automaton is deterministic and co-deterministic.
  \end{itemize}
  By the construction of the automaton, there is $i \in \mathbb{Z}$
  such that $t(z_i) = 0$ where $t(z_i)$ denotes the terminal state of
  the edge $z_i$. Thus $t(\bar{z_i}) \neq 0$ by (3).

  We will show that $t(z_i) = 0$ and $t(\bar{z_i}) \neq 0$ implies:
  \begin{itemize}
  \item{(4a)} $t(\bar{z_i}) \in\{\ldots,-2,-1\}$.
  \item{(4b)} There is an $s > 0$ with $t(z_{i+s}) = 0$ and
    $t(\bar{z}_{i+k})  \neq 0$ for all $k \in \{0, \ldots, s\}$.
  \end{itemize}
  Repeating the argument (4b) implies that $t(\bar{z}_{i+k}) \neq 0$
  for all $k \geq 0$. This leads to a contradiction since (4a) and the
  graph structure forces $t(\bar{z_{i+k}})$ to be the vertex $0$ for some
  $k \geq 0$. This proves that the automaton is strongly unambiguous.

  We now prove (4a) and (4b). Without loss of generality we may assume
  that $i = 0$.
Thus the label of $\cdots z_{-1}z_0$ is $v$ and the label of $z_1z_2
\cdots$ is $(wu)^{k_1}w$ with $k_1 = m_1$. The same holds for
$\bar{z}$.

If $t(\bar{z_0}) \in \{1, 2, \ldots \}$, then
\begin{itemize}
  \item either $dwu$ (the label
of $\bar{z_0}\bar{z_1} \cdots$) is a factor of some $wuwu$ at the
positive vertices, which is impossible by (1b),
\item or $uwb$ (the labels of the first edge of a
connecting $v$-path plus the edges connecting previous positive
vertices) is a factor of $uwuw$ inside $(wu)^{k_1}$, which is
impossible by (1a).
\end{itemize}

If $t(\bar{z_0})$ is a vertex inside some $v$-path connecting the
positive to the negative vertices, then $dwu$ (the labels of the last
edge of the connecting $v$-path plus the following ones connecting the
negative vertices) is a factor of $wuwu$ inside $(wu)^{k_1}$ since
$k_1 = m_1$ is large, contradicting (1b).
Thus $t(\bar{z_0}) \in \{-1, -2, \ldots \}$ which proves (4a).

We now prove (4b). By (4a) we know that $t(\bar{z_0}) \in \{-1, -2,
\ldots \}$. The label of $\bar{z_1}\bar{z_2} \cdots$ starts with
$(wu)^{k_1}$ since the label of $z_1z_2 \cdots$ does.
By $k_1 = m_1$ and (2) we have $t(\bar{z_0}) \leq N_1$.
Since the labels to the left of $0$ are $uwvw_{-1}v$, by $(1)$
the label of $\bar{z_1}\bar{z_2} \cdots$ have to leave the sequence
$(wu)^\infty$
at some time $n$, at the latest when $\bar{z}$ reaches the edge of the
first symbol of $v$ in $vw_{-1}v$. Thus, at time $n$ the path
$\bar{z}$ is still on the negative vertices and at least $|w_{-1}v|$
steps from the vertex $0$. At the same time $n$ the path $z$ has to
start a $v$-path connecting the positive and negative vertices. Let
$N_i$ be the terminal vertex of this $v$-path. Since $|v| <
|w_{-1}v|$, at the time $z$ reaches $N_i$ the path $\bar{z}$ is still
on the negative vertices. After the vertex $N_i$ the path $z$ has to
read the word $w(uw)^{m_i-1}$. By (2) this shows that at this time the
path $\bar{z}$ is not only on the negative vertices but in fact in a
vertex less or equal to $N_i$. By (3) it must be a vertex strictly
less than $N_i$. This show that $\bar{z}$ is "to the left of" $z$, \ie
when $z$ eventually reaches vertex $0$ (which it must since there is
no path starting in a negative vertex and avoiding 0), the path
$\bar{z}$ will not have reached $0$. This proves (4b).  
\end{proof}
The   question of whether any coded shift
can be unambiguously coded is raised in \cite{BurrDasWolfYang2020}.
We obtain easily a positive answer using Theorem~\ref{theoremFiebigRemark1.8}.
\begin{corollary}
  Every coded shift is unambiguously coded.
\end{corollary}
\begin{proof}
  Every coded shift is recognized by a deterministic,
  co-deterministic and strongly unambiguous automaton by the previous theorem. The set of
  first returns to some state $q$ of this automaton
  (that is, the labels of paths from $q$ to itself
  which do not pass by $q$ inbetween) defines a prefix code $C$
  such that the shift is unambiguously coded by $C$.
\end{proof}

This result can be used to compute the topological entropy
of a coded shift or its topological pressure
(see~\cite{BurrDasWolfYang2020} and \cite{Pavlov2018}).

\section{Synchronized shifts}
A word $w\in C^*$ is \emph{synchronizing} for a prefix code $C$
if for every $u,v\in A^*$, one has
\begin{equation}
  uwv\in C^*\Rightarrow uw,v\in C^*.\label{eqSynchro}
\end{equation}
A prefix code $C$ on the alphabet $A$
is \emph{synchronized} if there is a synchronizing word.
For an introduction to the notions
concerning codes, see \cite{BerstelPerrinReutenauer2009}.
A shift space  is
said to be a \emph{synchronized coded shift}
it it can be defined by a synchronizing prefix code.

As a closely related notion, a
word $w$ is a \emph{constant} for a language $L$
if it is a factor of $L$ and
if for every $u,v,u',v'\in A^*$, one has
\begin{displaymath}
  uwv,u'wv'\in L\Leftrightarrow uwv',u'wv\in L.
\end{displaymath}
Thus a word of $C^*$ is a constant for $C^*$ if and only if it is synchronizing.
A word $w$ is a constant for $L$ if and only if there
is a path labeled $w$ in the minimal automaton of $L$
and if all these paths end in  the same state.

When $L$ is a factorial language, the definition of a constant takes
a simpler form. Indeed, $w$ is a constant if and only if
\begin{equation}
  uw,wv\in L\Rightarrow uwv\in L\label{eqMagic}
\end{equation}
for every $u,v\in A^*$. Indeed it  is clear that  a constant
satisfies \eqref{eqMagic}. Conversely, if $w$ satisfies \eqref{eqMagic}
for all $u,v\in A^*$, assume that $uwv,u'wv'\in L$. Since
$L$ is factorial, we have also $uw,wv'\in L$ and thus $uwv'\in L$
by \eqref{eqMagic}. The proof that $u'wv\in L$ is similar.
Condition \eqref{eqMagic} is the one used to
define \emph{intrinsically synchronizing words} for shift
spaces (see \cite[Exercise 3.3.4]{LindMarcus2021}).

The following property gives a characterization of synchronized
shifts independant of the prefix code used to code the shift.
A deterministic
automaton $\A=(Q,i,i)$ is synchronized if
it is strongly connected and there exists a word $w$
such that $\Card\{p\cdot w\mid p\in Q \}=1$.

A stongly connected component $R\subset Q$ of an automaton $\A$
is said to be \emph{maximal} if for every edge $r\edge{a}s$
with $r\in R$, one has $s\in R$. The following statement is proved in
\cite{FiebigFiebig1992}.

\begin{proposition}\label{propositionCharacSynchro}
  An irreducible shift space $X$ is a synchronized coded shift
  if and only if the minimal automaton of $\cL(X)$
  has a unique maximal strongly connected component which is synchronized.
\end{proposition}
\begin{proof}
  Let $\A=(Q,i,T)$ be the minimal automaton of $\cL(X)$.
  Assume first that $X$ is coded by a synchronizing
  prefix code $C$.  Let $w\in C^*$ be a synchronizing word for $C$.
  Since $X$ is irreducible, for every $u\in \cL(X)$ there
  is a word $v$ such that $uvw\in\cL(X)$. Let us show
  that $i\cdot uvw=i\cdot w$. Indeed, note first that since $uvw\in \cL(X)$,
  there exist words $p,s$ such that $puvws\in C^*$. Since $w$ is synchronizing,
  this implies that $puvw\in C^*$. Assume now that $uvwt\in \cL(X)$.
  Then $wt\in\cL(X)$. Conversely, if $wt\in\cL(X)$, there are words $q,r$
  such that $qwtr\in C^*$. Since $w$ is synchronizing, we have
  $tr\in C^*$. Thus $(puvw)(tr)$ is in $C^*$ and thus $uvwt\in\cL(X)$.
  This shows that the strongly connected component of $i\cdot w$
  is the unique maximal strongly component of $\A$ and also that
  it is a synchronized automaton.

  Conversely, if $\A$ has
  a unique maximal strongly connected component $M\subset Q$ which is synchronized,
  let $q$ be an element of $M$ and let $C$ be the set
  of labels of paths from $q$ to $q$ which do not pass by $q$
  in between. Let  $w$ be a synchronizing word for $M$
  such that all paths labeled $w$ end in $q$.
 It is easy to see that $w$ can be extended in a synchronizing word for $C$.
  \end{proof}

The following statement is well known (see~\cite[Proposition 3.3.16]{LindMarcus2021}). 
\begin{proposition}\label{propositionSynchronized}
Every irreducible sofic shift is a synchronized coded shift.
\end{proposition}
The proof relies on the following statement.
\begin{lemma}\label{lemmaFischer}
  A deterministic strongly connected automaton $\A=(Q,i,i)$
  is synchronized if and only if
  there exists a word $w$ such that the set $I(w)=\{q\in Q\mid p\edge{w}q\}$
  is  finite and nonempty.
\end{lemma}
\begin{proof}
    Let $R$ be the set of finite nonempty subsets of $Q$ of the form $I(wu)$
  for $wu\in A^*$ which are of finite nonzero minimal cardinality.
  By assumption, this set is not empty. For every $I=I(wu)\in R$ and every
  $x\in \cL(X)$, there is a word $v$ such that
  $wuvx\in\cL(X)$ and consequently
  $I\cdot vx=I(wuvx)\in R$. Thus $\cL(X)$ is  the set of labels
  of  paths
  in the automaton $\A'=(R,I(w),I(w))$.
  Since $\A'$ is a synchronized automaton, this completes the proof.
  \end{proof}
%\begin{proof}
%Let $X$ be 
%  \end{proof}
Proposition~\ref{propositionSynchronized} is an easy consequence
of Lemma~\ref{lemmaFischer}.

\begin{example}
  The code $C=\{a,bb\}$ is synchronized because $a$ is a synchronizing
  word. This shows that the even shift is a synchronised coded shift.
\end{example}

The following statement is a particular case of Theorem~\ref{theoremFiebigRemark1.8}. We give an independent proof with a different and substantially
simpler construction. Instead of building an entirely new automaton,
as in the proof of Theorem~\ref{theoremFiebigRemark1.8}, we
modify the automaton in a way that preserves its structure
(for example, if the first automaton is finite, the new automaton
is also finite).
\begin{theorem}\label{theoremMain}
Every synchronized coded shift is unambiguously coded.
\end{theorem}
\begin{proof}
  Let $X$ be a coded shift defined by a synchronized prefix code $C$.
  Since there are synchronizing words for $C$, there
  are constants for $C^*$.
  Let $w\in A^*$ be a constant for  $C^*$
  and let $n$ be the length of $w$. 

  We consider the following automaton $\A$. The set of states
  $Q$ is the set of pairs $(u,p)$ formed of a word $u$ of length $n$ in $\cL(X)$
  and an element $p$
  of the set $P$ of states of the minimal automaton of $C^*$.
  Next, set  $(u,p)\cdot a=(v,p\cdot a)$
  where $v$ is such that $ua=bv$ for some $b\in A$.
 Since $w$ is a constant,
 there is a  state $(w,q_w)$ in $Q$ such that a path ends in $(w,q_w)$
 if and only if its label ends with $w$.

  Let $C'$ be the set of labels of simple paths from $(w,q_w)$
  to itself (such a path is simple if it does not pass by $(w,q_w)$
  in between). Then $X$ is coded by $C'$. Indeed, let $u,v$ be
  words with  $u$ ending with $w$  such
  that there is a path $i\edge{u}q_w\edge{v}i$ where $i$ is the
  initial and terminal state of $\A(C^*)$. If $c\in C^*$,
  then $q_w\edge{v}i\edge{cu}q_w$ and thus $vcu\in C'^*$.
  Conversely, if $c\in C'^*$, it is the label of path in $\A(C^*)$
  and thus is a factor of $C^*$.
Thus the factors of
  $C^*$ and $C'^*$ are the same.

  Consider an infinite
  path $\cdots q_{-1}\edge{a_{-1}}q_0\edge{a_0}q_1\cdots$
  with label $x=\cdots a_{-1}a_0a_1\cdots$ in $G$. We have
  $q_i=(w,q_w)$ if and only if the left infinite sequence
  $\cdots a_{i-2}a_{i-1}$ ends with $w$.
  
  It follows  that $X$ is unambiguously coded by $C'$
  since the sequence $c=(c_n)$
  corresponds to the labels of the paths between consecutive occurrences
  of $(w,q_w)$, with $c_0$ ending at the least $q_i=(q,q_w)$ with $i\ge 1$.
  The unique exponent $k$ with $0\le k<|c_0|$
  such that $x=\varphi^k(c)$ is then $k=|c_0|-i$.
\end{proof}
We note the following corollary.
\begin{corollary}\label{corollarySofic}
  Every irreducible sofic shift is unambigously coded by
  a rational prefix code.
\end{corollary}
Indeed, if $X$ is an irreducible sofic shift, it is synchronized
by Proposition~\ref{propositionSynchronized}. The prefix
code $C'$ build in the proof of Theorem~\ref{theoremMain}
is rational.

We illustrate the proof on two examples.
\begin{example}
  Let $X$ be the even shift, which is coded by $C=\{a,bb\}$.
  The letter $a$ is synchronizing for $C$ and the prefix code $C'=(bb)^*a$
  of Example~\ref{exampleEvenUnambiguous} is the result of the construction in the proof
  of Theorem~\ref{theoremMain}.
\end{example}
\begin{example}\label{exampleabba}
  Consider the shift coded by $C=\{ab,ba\}$. The minimal automaton
  of $C^*$ is represented in Figure~\ref{figureabba}.
  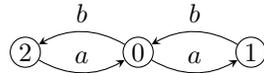
\begin{figure}[hbt]
\centering
\tikzset{node/.style={circle,draw,minimum size=0.4cm,inner sep=0.4pt}}
\begin{tikzpicture}
  \node[node](2)at(0,0){$2$};
  \node[node](0)at(1.5,0){$0$};
  \node[node](1)at(3,0){$1$};
  \draw[above,->,>=stealth,bend right](0)edge node{$b$}(2);
  \draw[above,->,>=stealth,bend right](2)edge node{$a$}(0);
  \draw[above,->,>=stealth,bend right](0)edge node{$a$}(1);
  \draw[above,->,>=stealth,bend right](1)edge node{$b$}(0);
\end{tikzpicture}
\caption{The minimal automaton of $\{ab,ba\}^*$.}\label{figureabba}
\end{figure}
  The word $w=bb$ is a constant since all paths labeled $bb$
  end in $2$. The automaton $\A$ built in the proof is
  represented in Figure~\ref{figureGraph}.
  \begin{figure}[hbt]
\centering
\tikzset{node/.style={circle,draw,minimum size=0.4cm,inner sep=0.4pt}}
\begin{tikzpicture}
  \node[node](aa)at(0,0){$aa,1$};
  \node[node](ab0)at(1.5,1){$ab,0$};\node[node](ba1)at(1.5,2.5){$ba,1$};
  \node[node](ba0)at(1.5,-1){$ba,0$};
  \node[node](ab2)at(1.5,-2.5){$ab,2$};
  \node[node](bb)at(3,0){$bb,2$};

  \draw[above,->,>=stealth,bend left](aa)edge node{$b$}(ab0);
  \draw[left,->,>=stealth,bend left](ab0)edge node{$a$}(ba1);
  \draw[right,->,>=stealth,bend left](ba1)edge node{$b$}(ab0);
  \draw[above,->,>=stealth,bend left](ab0)edge node{$b$}(bb);
  \draw[above,->,>=stealth,bend left](bb)edge node{$a$}(ba0);
  \draw[right,->,>=stealth,bend left](ba0)edge node{$b$}(ab2);
  \draw[left,->,>=stealth,bend left](ab2)edge node{$a$}(ba0);
  \draw[above,->,>=stealth,bend left](ba0)edge node{$a$}(aa);
\end{tikzpicture}
\caption{The automaton $\A$.}\label{figureGraph}
  \end{figure}
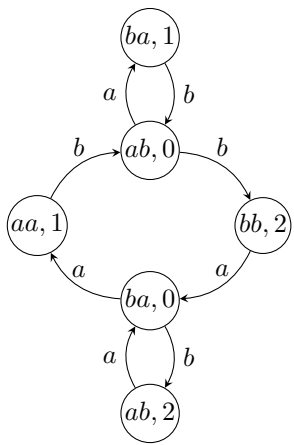
  
  The code $C'$ of first returns to $(bb,2)$ is
  \begin{displaymath}
    C'=a(ba)^*ab(ab)^*b.
    \end{displaymath}
\end{example}
A  code $C$ on the alphabet $A$
is \emph{circular} if for every $u,v\in A^*$ one has
\begin{equation}
  uv,vu\in C^*\Rightarrow u,v\in C^*.\label{eqCircular}
\end{equation}
If  the shift $X$ coded by $C$ is unambiguously coded, then $C$
is a circular code. Indeed, if $uv,vu$ are in $C^*$ although $u,v$ are not,
the bi-infinite sequence $(uv)^\infty$ has two factorizations
in words of $C$. To see this in more detail, set $uv=c_1\cdots c_n$
and $vu=d_1\cdots d_m$ with $c_i,d_i\in C$. Then we have a factorization
$c_i=ps$ with $p$ nonempty such that $v=sc_{i+1}\cdots c_n$
and $u=c_1\cdots c_{i-1}p$.
Then the equality
\begin{displaymath}
  (c_ic_{i+1}\cdots c_nc_1\cdots c_{i-1})^\infty=S^k(vu)^\infty
\end{displaymath}
with $k=|p|$ forces $k=0$ and thus $u,v\in C^*$.

The fact that every irreducible sofic shift is coded by a circular
code is proved in \cite{BealPerrin1986}. By the above remark,
this follows from Theorem~\ref{theoremMain}.

It is possible to prove Theorem~\ref{theoremMain} with
a different construction using the notion of \emph{state splitting}
(see \cite[Proposition 2.4]{Beal1993}). We do not develop
this proof but we show its steps on the shift of Example \ref{exampleabba}.
\begin{example}

  \begin{figure}[hbt]
\centering
\tikzset{node/.style={circle,draw,minimum size=0.4cm,inner sep=0.4pt}}
\begin{tikzpicture}
  %1
  \node[node](3)at(0,0){$3$};\node[node](1)at(1.5,0){$1$};\node[node](2)at(3,0){$2$};

  \draw[above,->,>=stealth,bend right](1)edge node{$b$}(3);
  \draw[above,->,>=stealth,bend right](3)edge node{$a$}(1);
  \draw[above,->,>=stealth,bend left](1)edge node{$a$}(2);
  \draw[above,->,>=stealth,bend left](2)edge node{$b$}(1);

  %2
  \node[node](3)at(4,0){$3$};\node[node](1)at(5.5,.5){$1$};
  \node[node](1')at(5.5,-.5){$1'$};\node[node](2)at(7,0){$2$};

\draw[above,->,>=stealth,bend right](1)edge node{$b$}(3);
\draw[above,->,>=stealth,bend left](1')edge node{$b$}(3);
  \draw[above,->,>=stealth,bend left](3)edge node{$a$}(1');
  \draw[above,->,>=stealth,bend left](1)edge node{$a$}(2);
  \draw[above,->,>=stealth,bend left](2)edge node{$b$}(1);
  \draw[above,->,>=stealth,bend right](1')edge node{$a$}(2);
  %3
  \node[node](3)at(0,-2){$3$};
  \node[node](3')at(0,-3){$3'$};
  \node[node](1)at(1.5,-2){$1$};
  \node[node](1')at(1.5,-3){$1'$};
  \node[node](2)at(3,-2.5){$2$};

\draw[above,->,>=stealth,bend right](1)edge node{$b$}(3);
\draw[above,->,>=stealth,bend left](1')edge node{$b$}(3');
\draw[above,->,>=stealth,bend left](3)edge node{$a$}(1');
\draw[above,->,>=stealth,bend left](3')edge node{$a$}(1');
  \draw[above,->,>=stealth,bend left](1)edge node{$a$}(2);
  \draw[above,->,>=stealth,bend left](2)edge node{$b$}(1);
  \draw[above,->,>=stealth,bend right](1')edge node{$a$}(2);
  %4
  \node[node](3)at(4,-2){$3$};
  \node[node](3')at(4,-3){$3'$};
  \node[node](1)at(5.5,-2){$1$};
  \node[node](1')at(5.5,-3){$1'$};
  \node[node](2)at(7,-2){$2$};
  \node[node](2')at(7,-3){$2'$};

\draw[above,->,>=stealth,bend right](1)edge node{$b$}(3);
\draw[above,->,>=stealth,bend left](1')edge node{$b$}(3');
\draw[above,->,>=stealth,bend left](3)edge node{$a$}(1');
\draw[above,->,>=stealth,bend left](3')edge node{$a$}(1');
  \draw[above,->,>=stealth,bend left](1)edge node{$a$}(2);
  \draw[above,->,>=stealth,bend left](2)edge node{$b$}(1);
  \draw[above,->,>=stealth,bend right](1')edge node{$a$}(2');
  \draw[above,->,>=stealth,bend left](2')edge node{$b$}(1);

\end{tikzpicture}
\caption{The state splitting.}\label{figureStateSplitting}
  \end{figure}
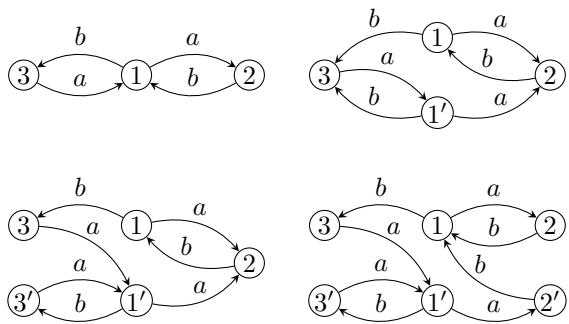
Set $C=\{ab,ba\}$ as in Example~\ref{exampleabba}.
We start with the minimal automaton of $C^*$ shown in Figure~\ref{figureStateSplitting} on the left. We split state $1$ into two states $1$ and $1'$
having the same output but $1$ receives the input edge from
$2$ and $1'$ the input edge from $3$. The result is shown in Figure~\ref{figureStateSplitting} at the top on the right. Then, we split state $3$ into states $3$
and $3'$ as indicated in Figure \ref{figureStateSplitting}
at the bottom on the left.
As a result, a path ends at state $3$ if and only if its
label ends with $bb$. Finally, we spilt $2$ into
$2$ and $2'$. The result is the same as in Example~\ref{exampleabba}.
\end{example}

There exist coded shifts which are unambiguously coded by a prefix code $C$
although $C$ is
not synchronized, as shown by the following example.
\begin{example}\label{exampleDyck}
  Let $A=\{a,b,\bar{a},\bar{b}\}$ and let $D$ be the unique language on $A$ such
  that
  \begin{displaymath}
    D=aD^*\bar{a}\cup bD^*\bar{b}
    \end{displaymath}

    The prefix code $D$ is not synchronized. Indeed, for every $d\in D^*$,
    one has $ad\bar{a}\in D$ although $\bar{a}$ is not in $D^*$. The
    coded shift defined by $D$ is unambiguously coded. Indeed,
    no proper nonempty suffix of an element of $D$ can
    be a prefix of an element of $D^*$. This coded shift
    is known as a \emph{Dyck shift} (see \cite{krieger1974}
    or~\cite{Beal2016}). The fact that $D$ is a circular code
    is proved in~\cite{deLucaRestivo1980}.
\end{example}

A  code $C$ is \emph{very thin} if there is a word $c\in C^*$
such that $c$ is not a factor of $C$. Every rational code is very
thin (see \cite[Theoem 9.4.1]{BerstelPerrinReutenauer2009}).
The prefix code $D$ of Example~\ref{exampleDyck} is not
very thin. Indeed, every $d\in D^*$ is a factor of $ad\bar{a}\in D$.

\begin{theorem}\label{theoremCodedVeryThin}
  A coded shift defined by a very thin prefix code
  is synchronized.
\end{theorem}
\begin{proof}
  Assume that $X$ is coded by a very thin prefix code $C$.
  Let $\A=(Q,i,i)$ be the minimal automaton of $C^*$.
  For $w\in A^*$, set
  \begin{displaymath}
    I(w)=\{q\in Q\mid p\cdot w=q\mbox{ for some $p\in Q$}\}.
  \end{displaymath}
  Let $w\in C^*$ be a word which is not a factor of $C$.
  Then the set $I(w)$
    is finite. Indeed, assume that $p\cdot w=q$. Let $u,v$ be such
  that $i\cdot u=p$ and $q\cdot v=i$. Then $uwv\in C^*$ forces
  $w=rs$ with $ur,sv\in C^*$ and thus $p\cdot w=i\cdot s$. This shows that $I(w)$ is contained
  in the finite set $\{i\cdot s\mid \mbox{ $s$ is a suffix of $w$}\}$.

By Lemma~\ref{lemmaFischer}, this implies that $X$ is synchronized.
\end{proof}

\begin{example}
  A $\beta$-shift $X_\beta$ is synchronized if and only if the orbit
  of $d_\beta(1)$ is not dense in $X_\beta$.
  One of the implications of this result, due to \cite{Bertrand-Mathis1986}, can be proved easily using
  Theorem~\ref{theoremCodedVeryThin}. Indeed, let $C$
  be the prefix code formed by the labels of simple paths
  from $0$ to $0$ in the automaton of Figure~\ref{figureBetaShift}. Let
  $w$ be a word in $\cL(X_\beta)$ which does not appear in $d_\beta(1)$,
  and let $a$ be a letter such that $wa\in\cL(X_\beta)$. Then
  $wa$ is a factor of $C^*$ but not a factor of $C$ and thus $C$ is very thin.

\end{example}
\bibliographystyle{plain}
\bibliography{coded.bib}
\end{document}